\documentclass[preprint,12pt]{elsarticle}
\usepackage{amssymb}
\journal{Discrete Applied Mathematics}

\usepackage{graphicx}
\usepackage[ruled,noend]{algorithm2e}
\usepackage{amsmath, amsthm, amssymb}
\usepackage{verbatim}
\newtheorem{theorem}{Theorem}

\newtheorem{corollary}[theorem]{Corollary}

\def\today{August 5, 2010} 

\begin{document}

\begin{frontmatter}

\title{Minimum $k$-path vertex cover}

\author[maribor,ljubljana]{Bo\v stjan Bre\v sar}
\ead{bostjan.bresar@uni-mb.si}

\author[umvupjs]{Franti{\v s}ek Kardo{\v s}}
\ead{frantisek.kardos@upjs.sk}

\author[icsupjs]{J{\'a}n Katreni{\v c}}
\ead{jan.katrenic@upjs.sk}

\author[icsupjs]{Gabriel Semani{\v s}in}
\ead{gabriel.semanisin@upjs.sk}

\address[maribor]{Faculty of Natural Science and Mathematics, University of Maribor, Slovenia}
\address[ljubljana]{Institute of Mathematics, Physics and Mechanics, Ljubljana, Slovenia}

\address[umvupjs]{
     Institute of Mathematics,
     P.J. {\v S}af{\'a}rik University,
     Ko{\v s}ice, Slovakia      
}

\address[icsupjs]{
     Institute of Computer Science,
     P.J. {\v S}af{\'a}rik University,
     Ko{\v s}ice, Slovakia 
}

\begin{abstract}
A subset $S$ of vertices of a graph $G$ is called a {\em $k$-path vertex cover} 
if every path of order $k$ in $G$ contains at least one vertex from $S$. 
Denote by $\psi_k(G)$ the minimum cardinality of a $k$-path vertex cover in $G$.
It is shown that the problem of determining $\psi_k(G)$ is NP-hard for each $k\geq2$, while for trees the problem can be solved in linear time. 
We investigate upper bounds on the value of $\psi_k(G)$ and provide several estimations and exact values of $\psi_k(G)$. We also prove that $\psi_3(G)\le (2n+m)/6$, for every graph $G$ with $n$ vertices and $m$ edges. 
\end{abstract}


\end{frontmatter}

\section{Introduction and motivation}

In this paper we consider finite graphs without loops and multiple edges and use standard graph theory notations (see e.g.~\cite{West96}). In particular, 
by the {\em order of a path $P$} we understand the number of vertices on $P$ while the {\em length of a path} is the number of edges of $P$. 

We introduce a new graph invariant that generalizes the intensively studied concept of vertex cover.
Our research is motivated by the following problem \cite{Novotny2010} related to secure communication in wireless sensor networks (shortly WSN). The topology of WSN can be represented by a graph, in which vertices represent sensor devices and edges represent communication channels 
between pairs of sensor devices.
Traditional security techniques cannot be applied directly to WSN, because sensor devices are limited in their computation, energy, communication capabilities. Moreover, they are often deployed in accessible areas, where they can be rather easily captured by an attacker.
In general, a standard sensor device is not considered as tamper-resistant and it is undesirable to make all devices of a sensor network tamper-proof due to increasing cost. Therefore, the design of WSN security protocols has become a challenge in security research.

We focus on the Canvas scheme \cite{Goll01,Vogt,Novotny2010,NovotnyF2010} which should provide
data integrity in a sensor network.
The $k$-generalized Canvas scheme \cite{Novotny2010} guarantees data integrity under the assumption 
that at least one node which is not captured exists on each path of the length $k-1$ in the communication graph. The scheme combines  the properties of cryptographic primitives and the network topology.
The model distinguishes between two kinds of sensor devices -- {\em protected} and {\em unprotected}. The attacker is unable to copy secrets from a protected device. This property can be realized by making
the protected device tamper-resistant or placing the protected device on a safe
location, where capture is problematic. On the other hand, an unprotected device
can be captured by the attacker, who can also copy secrets from the device
and gain control over it. During the deployment and initialization of a sensor
network, it should be ensured, that at least one protected node exists on each
path of the length $k-1$ in the communication graph \cite{Novotny2010}. The problem to minimize the cost of the network by minimizing the number of protected vertices is formulated in \cite{Novotny2010}.

Formally, let $G$ be a graph and let $k$ be a positive integer. A subset of vertices $S\subseteq V(G)$ is called a {\em $k$-path vertex cover} if every path of order $k$ in $G$ contains at least one vertex from $S$. We denote by  $\psi_k(G)$ the minimum cardinality of a $k$-path vertex cover in $G$. 

Clearly, $2$-path vertex cover corresponds to the well know vertex cover (a subset of vertices such that each edge of the graph is incident to at least one vertex of the set).
Therefore $\psi_2(G)$ is equal to the size of the minimum vertex cover of a graph $G$. Moreover, the value of $\psi_3(G)$ corresponds to the so-called dissociation number of a graph \cite{Boliac,Goring,Yan} defined as follows. A subset of vertices in a graph $G$ is called a {\em dissociation set} if it induces a subgraph with maximum degree $1$. The number of vertices in a maximum cardinality dissociation set in $G$ is called the {\em dissociation number} of $G$ and is denoted by $diss(G)$. Clearly, $\psi_3(G)=|V(G)|-diss(G)$.

A related colouring is the {\em $k$-path chromatic number}, which is
the minimum number of colours that are necessary for colouring the
vertices of $G$ in such a way that each colour class forms a $P_k$-free set
\cite{Aki89,FrBu01,JoSa06,Jin95,MynBr85}, and it is easy to see that $\psi_k(G)\le \frac{\chi_k(G)-1}{\chi_k(G)}|V(G)|$.

\bigskip

Since the minimum vertex cover problem is NP-hard \cite{Karp1972}, it is not surprising that so is the problem of determining $\psi_k$ for each $k\ge 3$. We provide details in Section~\ref{NPC} -- actually, we reduce the minimum vertex cover problem to the minimum $k$-path vertex cover problem.

However, since the question whether there is a $k$-path vertex cover of size at most $t$ can be expressed in the monadic second order logic, by famous Courcelle's theorem \cite{Cour}, the minimum $k$-path vertex cover problem can be solved in linear time on graphs with bounded treewidth, e.g.~trees, series-parallel graphs, outerplanar graphs, etc. 
In Section~\ref{trees} we determine the exact value of $\psi_k$ for trees and present a linear time algorithm which returns an optimal solution for trees. Then Section~\ref{outer} is devoted to outerplanar graphs. We present a tight upper bound for $\psi_3(G)$.

In Section~\ref{sparse} we provide several estimations for the size of minimum $k$-path vertex cover on degree of its vertices. Finally, we prove that $\psi_3(G)\le (2n+m)/6$, for every graph $G$ with $n$ vertices and $m$ edges.

\section{NP-completeness}
\label{NPC}

In this section we prove that the problem of determining $\psi_k$ is NP-complete for every $k\ge 2$.
For $k=2$ this equals to the vertex cover problem (shortly VCP)
which is known to be NP-complete \cite{Karp1972}. For an arbitrary fixed integer $k\geq 2$ the decision version of the $k$-path vertex cover problem can be stated as follows:

\medskip

{\bf $k$-Path Vertex Cover Problem} ($k$-PVCP)

INSTANCE: Graph $G$ and a positive integer $t$.

QUESTION: Is there a $k$-path vertex cover $S$ for $G$ of size at most $t$?

\bigskip

\begin{theorem}
For any fixed integer $k\ge 2$ the $k$-Path Vertex Cover
Problem is NP-complete.
\label{th:np}
\end{theorem}
\begin{proof}
First, it is not hard to see that $k$-PVCP is in NP. Note that there are at most $O(n^k)$ paths of length $k-1$ in $G$. Thus for a given set $S$ we can decide in polynomial time whether all paths of length $k-1$ are secured by the vertices from $S$.

Note that 2-PVCP coincides with VCP. For $k>2$ we will reduce VCP to $k$-PVCP. 
Let $G$ be an arbitrary graph. Let $G'$ be a graph obtained from $G$ by joining  a path with $\lfloor \frac
{k-1}{2}\rfloor$ new vertices to every vertex of $G$ (i.e.~for $k=3$ or $k=4$ we
attach a leaf to each vertex; for $k=5$ or $k=6$ a path on two
vertices is attached to each vertex of $G$; etc.). We call the vertices of $G$ {\em original vertices},
and others are called {\em new}. In order to finish the proof it suffices to show that the size of a minimum vertex cover 
in $G$ is equal to $\psi_k(G')$.

Suppose that $G'$ has a $k$-path vertex cover $S'$. Suppose that $S'$ contains also a new vertex $x$ (i.e.~it lies on an attached path). Let $x'$ be an original vertex that has the smallest distance from $x$. Note that $x$
secures only the paths that go through the vertex $x'$, since the length of the attached path is less than $k$. 
Hence, if $x$ is replaced in $S'$ by $x'$, all paths remain secured. 
 By consecutive application of this procedure for each new vertex $x$ we obtain a set $S$ comprising
only of the original vertices. 
Clearly, $|S|\leq |S'|$, and by the
 procedure described above, $S$ is also a $k$-path vertex cover. We claim
that $S\subseteq V(G)$ forms a vertex cover in $G$.

 Suppose to the contrary that there is an edge $xy$ in $G$ whose  endvertices $x$ and $y$ are not in $S$.
Consider the path $P$ in $G'$, composed by the path,
attached to $x$, vertices $x$ and $y$, and the path attached to $y$.
Then clearly $P$ contains no vertex from $S$, and it has $2\lfloor \frac
{k-1}{2}\rfloor+2\geq k$ vertices, which is a contradiction. Thus
$S$ is a vertex cover in $G$ of size at most $|S'|$. 

Conversely, by a similar reasoning we can show that a vertex cover of $G$ yields a $k$-path vertex cover set in $G'$ of the same size.
\end{proof}

The reduction described in the proof of Theorem \ref{th:np} also shows that for any approximation rate $r$ one can transform a polynomial time $r$-appro\-xi\-ma\-tion for the minimum $k$-path vertex cover problem
to a polynomial time $r$-approximation algorithm for the minimum vertex cover. 
Using a result of \cite{Dinur04}, we immediately infer that it is NP-hard to approximate $k$-PVCP, for $k>2$,
within a factor of $1.3606$, unless $P=NP$.

A well-known $2$-approximation algorithm which repeatedly puts vertices of an edge into the constructed vertex cover  and removes them from the graph, was discovered independently by F.~Gavril and M.~Yannakakis \cite{Cormen}. A polynomial time approximation with a better constant factor is not known for the vertex cover problem. Similarly, for the minimum $k$-path vertex cover problem we can construct an approximation algorithm in the following way: Find a path on $k$ vertices, put its $k$ vertices into the constructed set and remove them from the graph. Since  at least one of these $k$ vertices belongs to an optimal solution, this provides a $k$-approximation algorithm for the minimum $k$-path vertex cover with a polynomial running time while $k$ is considered as a fixed constant.
 Note that the currently best known randomized algorithm for finding a path on $k$ vertices runs in time $O(2^k n^{O(1)})$, where $n$ denotes the order of a graph (see \cite{Williams09}).

\section{Path vertex cover for trees}
\label{trees}

We begin our investigation with the class of trees, that present an important underlying 
communication topology for WSN. Courcelle's theorem \cite{Cour} guarantees the existence of a linear time algorithm, basically using dynamic programming. In this section, we describe such an algorithm in detail, and we use it then to derive a sharp upper bound 
$\psi_k(T) \le |V(T)|/k$ for an arbitrary tree $T$.

In order to simplify our consideration, consider that the input tree is rooted at a vertex $u$.  By {\em properly rooted subtree} we denote 
a subtree $T_v$, induced by a vertex $v$ and its descendants (with respect to $u$ as the root) 
and satisfies the following properties:
\begin{enumerate}
 \item $T_v$ contains a path on $k$ vertices;
 \item $T_v \smallsetminus v$ does not contain a path on $k$ vertices.
\end{enumerate}
The algorithm $\mathtt{PVCPTree}$ systematicaly searches for a properly rooted tree $T_v$, puts $v$ into a solution and removes $T_v$ from the input tree $T$.


\begin{function}[H]
  \caption{PVCPTree(T,k)}
  \KwIn{A tree $T$ on $n$ vertices and a positive integer $k$\;}
  \KwOut{A $k$-path vertex cover $S$ of $T$\;}
  \SetArgSty{textbb}   
   Form an arbirary vertex $u\in T$, make $T$ rooted in $u$\;
   $S := \emptyset$\;
   \While{$T$ contains a properly rooted subtree $T_v$ }
   {            
       $S := S \cup \{v\}$\;
       $T := T\smallsetminus T_v$\;
   }
   \Return $S$\;
\end{function}


\begin{theorem}\label{tree}
 Let $T$ be a tree and $k$ be a positive integer.
 The algorithm $\mathtt{PVCPTree}(T,k)$ returns an optimal $k$-path vertex cover of $T$ of size at most $\frac{|V(T)|}{k}$.
Therefore, $\psi_k(T) \le \frac{|V(T)|}{k}.$
\end{theorem}

\begin{proof}
First, we shall argue that $\mathtt{PVCPTree}$ returns an optimal solution. We prove this by induction on the number of vertices in the tree $T$.  If $T$ does not contain any path on $k$ vertices, the empty set is the optimal solution. Suppose $T$ contains a path on $k$ vertices. Let $T_v$ be a properly rooted subtree of $T$. Since any $k$-path cover of $T$ contains a vertex of $T_v$ it follows that $\psi_k(T)=\psi_k(T\smallsetminus T_v)+1$ and hence the result follows by induction.

To prove that the returning set $S$ contains at most $\frac{|V(T)|}{k}$ vertices, we argue that
each loop of the algorithm inserts one vertex into $S$ and removes from $T$ one properly rooted subtree having at least $k$ vertices. 
\end{proof}

Concerning the time complexity of the algorithm, it is straightforward to implement the algorithm $\mathtt{PVCPTree}$ such that the returning set $S$ in computed in linear time.

%

\section{Outerplanar graphs}
\label{outer}
In this section we study the 3-path vertex cover of outerplanar graphs.

\begin{theorem}
Let $G$ be an outerplanar graph of order $n$. Then $\psi_3(G)\le \frac{n}{2}$.
\label{th:out}
\end{theorem}
\begin{proof}
We prove the statement by induction on the number of vertices of $G$. 

Let $H$ be a maximal outerplanar graph such that $G$ is its subgraph satisfying $V(H)=V(G)$. (Recall that $H$ is \emph{maximal outerplanar} if $H$ is outerplanar, but adding any edge destroys that property.) It is easy to see that $H$ is 2-connected and all its inner faces are triangles.
Observe that every $k$-path vertex cover in $H$ is a $k$-path vertex cover in $G$, since $G$ is a subgraph of $H$. Therefore, it suffices to find a $3$-path vertex cover in $H$ of size at most $\frac{n}2$.

Obviously, if $H$ consists of a single triangle, then $\psi_3(H)=1\le \frac32 =\frac{n}2$. 

Assume $H$ has at least four vertices. Since $H$ is a 2-connected outerplanar graph, the closed trail bounding the outer face contains each vertex of $H$ precisely once, hence $H$ is Hamiltonian. Let $v_1,v_2,\dots,v_n$ be the cyclic ordering of vertices of $H$ along the Hamiltonian cycle. Colour a vertex $v_i$ white, if the degree of $v_i$ in $H$ is $2$, otherwise colour it black. Since all the inner faces of $H$ are triangles, there are no two consecutive white vertices, unless $H$ consists of a single triangle, which is excluded. Hence, the white vertices induce an independent set. The edge $v_iv_{i+1}$ is called \emph{good}, if it has a white endvertex, otherwise it is \emph{bad}.
If all the edges incident with the outer faces are good, it implies that $n$ is even, and half of the vertices are white. Then the set of all black vertices is a $3$-path vertex cover of size $\frac{n}2$, since the white vertices form an independent set.

Assume there is a bad edge, i.e.~there is at least one pair of consecutive black vertices, say $v_i,v_{i+1}$. 
We claim that there is an edge $v_iv_j$ of $H$ such that $v_i$, $v_{i+1}$, and $v_j$ form a triangular face 
adjacent to the outer face, all the three vertices $v_i$, $v_{i+1}$, and $v_j$ are black, and all the edges
$v_{i+1}v_{i+2}$,$v_{i+2}v_{i+3}$,\dots,$v_{j-1},v_j$ (indices taken modulo $n$) are good.

For the sake of contradiction, suppose that this is not the case. Let $e=v_{i}v_{i+1}$ be a bad edge. 
There is a unique triangular face of $H$ incident both with $v_{i}$ and $v_{i+1}$; let $v_{j}$ be the third vertex incident with the face. It is easy to see that $v_{j}$ must be black. The vertices $v_{i}$, $v_{i+1}$, and $v_{j}$ cut the Hamiltonian cycle into the edge $v_{i}v_{i+1}$ and two paths; let $\sigma(e)$ be the smaller of their lengths. 

Let $e_0=v_{i_0}v_{i_0+1}$ be the bad edge with $\sigma(e)$ minimal; let $v_{j_0}$ be the corresponding black vertex such that together with $v_{i_0}$ and $v_{i_0+1}$ they form a triangular face of $H$. Without loss of generality assume that $P=v_{i_0+1}$,\dots,$v_{j_0}$ is the path of length $\sigma(e)$. Observe that $P$ contains at least one white vertex, thus, it contains good edges. 
If all the edges $v_{i_0+1}v_{i_0+2}$,\dots,$v_{j_0-1}v_{j_0}$ are good, we are done. Otherwise, we find a bad edge $e_1=v_{i_1}v_{i_1+1}$ of $P$.
Again, there is a unique triangular face of $H$ incident both with $v_{i_1}$ and $v_{i_1+1}$; let $v_{j_1}$ be the third (black) vertex incident with the face. Since $\{v_{i_0+1},v_{j_0}\}$ is a 2-cut of $H$ separating the vertices $v_{i_0+2}$,\dots,$v_{j_0-1}$ from the rest of the graph, we have $v_{j_1}\in\{v_{i_0+1},\dots, v_{j_0}\}$. But then $\sigma(e_1)<\sigma(e_0)$, a contradiction with the choice of $e_0$.

Therefore, there is an edge $e=v_iv_j$ of $H$ such that $v_i$, $v_{i+1}$, and $v_j$ form a triangular face adjacent to the outer face, all the three vertices $v_i$, $v_{i+1}$, and $v_j$ are black, and all the edges $v_{i+1}v_{i+2}$,\dots,$v_{j-1}v_j$ are good. It means that on the path $P=v_{i+1},\dots,v_{j}$ black and white vertices alternate, thus, $\sigma(e)=2s$ for some integer $s$. 

Let $W$ be the set of white vertices on $P$, let $B$ be the set of black vertices on $P$, including $v_{i+1}$ and $v_j$. Obviously $|W|=s$ and $|B|=s+1$; vertices from $W$ induce an independent set in $H$ and $v_{i+1}$ is adjacent to precisely one of them.
Let $S'$ be a $3$-path vertex cover in the graph $H'=H\setminus\{v_i,v_{i+1},\dots,v_j\}$ of size at most $\frac{|V(G')|}2 = \frac{n-2s-2}{2}$ given by induction. Then $S=S'\cup \{v_{i}\}\cup(B\setminus\{v_{i+1}\})$ is a $3$-path vertex cover in $H$ of size at most $\frac{n-2s-2}{2}+1+s =\frac{n}2$.
\end{proof}

The bound $\frac{n}2$ in Theorem \ref{th:out} is the best possible, since it is easy to find outerplanar graphs with $\psi_3\ge \frac{n}2$.
Consider an arbitrary 2-connected outerplanar graph $G$. Let $v_1,v_2,\dots,v_n$ be the ordering of its vertices along the Hamiltonian cycle (the boundary of the outer face). For each edge $v_iv_{i+1}$ incident with the outer face ($v_{n+1}=v_1$), add a new vertex $u_i$ and two new edges $u_iv_i$ and $u_iv_{i+1}$. Let the resulting graph be $H$. It is easy to see that $H$ is a 2-connected outerplanar graph on $2n$ vertices. We claim that $\psi_3(H)\ge n$. Let $S$ be a $3$-path vertex cover in $H$. Divide the vertices of $H$ into $n$ pairs of the form $v_i,u_i$. 
Observe that if $u_i\notin S$ and $v_i\notin S$, then $u_{i-1}\in S$ and $v_{i-1}\in S$ (otherwise there is a path on 3 vertices in $H$ not covered by $S$). Therefore, the average number of vertices in $S$ is either at least $\frac12$ (if at least one vertex of a pair is in $S$), or $\frac24$ (if no vertex is in $S$, then in the previous pair both are). Altogether, $|S|\ge n$.

%
%

\section{Upper bounds on degree of vertices}
\label{sparse}

In this section we provide several upper bounds on path vertex cover based on degrees of a graph. First, recall   theorem of Caro \cite{Caro} and Wei \cite{Wei}, which states 
$$\psi_2(G) \le |V(G)| - \sum_{u\in V(G)}{\frac{1}{1+d(u)}}$$
for every graph G. Since the only graphs for which this is best-possible are the disjoint unions
of cliques, additional structural assumptions 
allow improvements
(see e.g.~\cite{Harant1,Harant3,Heckman1,Heckman2,Lowenstein}).

The problem of dissociation number of graph was also studied in \cite{Goring}, where G{\" o}ring et~al. proved
$$\psi_3(G) \le |V(G)| - \sum_{u\in G}{\frac{1}{1+d(u)}} - \sum_{uv\in E(G)}{\frac{2}{|N(u)\cup N(v)| (|N(u)\cup N(v)|-1)}}. $$

In what follows, we provide the following generalization of Caro-Wei theorem for $\psi_k(G)$.

\begin{theorem}
 For every graph $G$:
$$\psi_k(G) \le |V(G)| - \displaystyle \frac{k-1}{k} \sum_{u\in G}{\frac{2}{1+d(u)}}. $$
\end{theorem}
\begin{proof}
 Let us arbitrarily order the vertices of $G$, and let us start with $S$ being the empty set. One by one let us add
vertex $v_i$ to $S$ unless at least two of its neighbours are already there. The probability that $v_i$
will eventually land in $S$, is $\frac{2}{1 + d(v_i)}$, because it is the probability that in random
ordering of vertices of $G$, $v_i$ precedes each of its neighbours except for one. From this one can deduce
that the expected size of the set $S$ is at least $\sum_{v_i\in V(G)}{\frac{2}{(1 + d(v_i))}}$.
Since $S$ is a $1$-degenerated graph, $S$ is a forest. Using Theorem~\ref{tree} we get  $\psi_k(S) \le \frac{1}{k} |V(S)|$. Finally, to construct a $k$-path vertex cover for $G$ we put into solution $V(G)\smallsetminus S$ and all the vertices forming the minimum $k$-path vertex cover of $S$.
\end{proof}

In the rest of this section, we investigate the upper bounds to the $3$-path vertex cover based on maximum and average degree of a graph. First, recall a well known decomposition theorem of Lov{\'a}sz~\cite{Lovasz}:

\begin{theorem}
\label{Tlovasz}
If $s$ and $t$ are non-negative integers, and if $G$ is a graph with maximum degree at most $s + t + 1$, then the vertex set of $G$ can be partitioned into two sets which induce subgraphs of maximum degree at most $s$ and $t$, respectively.
\end{theorem}

The following is a straightforward generalization of Theorem \ref{Tlovasz} for $k$-sets (see e.g.~\cite{Cowen}).

\begin{theorem}\label{decomposition}
 For any graph $G$ of maximum degree $\Delta$, the vertex set of $G$ can be partitioned into $k$ sets which induce subgraphs of maximum degree at most $\frac{\Delta}{k}$. Moreover, this can be computed in running time $O(\Delta |E(G)|)$.
\end{theorem}

\begin{corollary}\label{cubic_half}
Let $G$ be a graph of maximum degree $\Delta$. 
Then  $$\psi_3(G)\le \frac{\lceil \frac{\Delta-1}{2} \rceil}{\lceil \frac{\Delta+1}{2} \rceil} |V(G)|.$$
Moreover, such a solution can be computed in running time $O(\Delta|E(G)|)$.
\end{corollary}

\begin{proof}
According to Theorem~\ref{decomposition}, the vertex set of $G$ can be partitioned into $\lceil \frac{\Delta+1}{2} \rceil$ sets which induce subgraphs of maximum degree at most $1$. 
Let $S$ denote the largest set of such a partition; clearly $|S|\ge \frac{|V(G)|}{\lceil \frac{\Delta+1}{2} \rceil}$. 
To create a $3$-path vertex cover set for $G$, we protect $V(G)\smallsetminus S$.
\end{proof}

\begin{corollary}\label{subcubic}
 There is a linear time algorithm which returns a $3$-path vertex cover of size at most
$\min(\frac{|V(G)|}{2},\frac{|E(G)|}{2})$, for arbitrary graph $G$ of maximum degree at most $3$.
\end{corollary}

\begin{proof}
 To construct a $3$-path vertex cover of size at most $|E(G)|/2$ one can repeatedly put into a solution a vertex of degree at least $2$.
 To construct a $3$-path vertex cover of size at most $|V(G)|/2$ for a graph $G$, $\Delta(G)\le 3$, one can use 
 Corollary~\ref{cubic_half}.
\end{proof}

Next we provide an algorithm which guarantees a tight upper-bound for $\psi_3$ based on the number of vertices and edges of a graph. 

\begin{algorithm}[H]
  \caption{$\mathtt{SPARSE_3}{(G)}$}
  \KwIn{A graph $G$ on $n$ vertices\;}
  \KwOut{$3$-path vertex cover $H$, $H\subseteq V(G)$\;}
  \SetArgSty{textbb}  
  $H := \emptyset$\;
  \While{$G$ contains any vertex $v$ of degree at least $4$}{          
     $H := H \cup \{v\}$\;
     Remove from $G$ the vertex $v$ and all edges incident with $v$\;
  }
  $H:= H\cup \mathtt{Solve\_SubCubic}(G)$\;
  \Return $H$\;
\end{algorithm}
\noindent
Here, $\mathtt{Solve\_SubCubic}(G)$ denotes the algorithm of Corollary~\ref{subcubic}.

\begin{theorem}
 Let $G$ be a graph having $n$ vertices and $m$ edges. The algorithm $\mathtt{SPARSE_3}(G)$ returns a $3$-path vertex cover $H$ of size at most $\frac{2n+m}{6}$.
\end{theorem}

\begin{proof}
The algorithm consists of two phases. The first phase repeatedly erases vertices of degree at least $4$. After this step, we get a graph $G$ of maximum degree $3$. In the second phase the algorithms uses the method 
$\mathtt{Solve\_SubCubic}(G)$.

Let $s$ denote the number of vertices (of degree at least $4$) protected in the first phase, let $d$ denote the number of vertices protected in the second phase.
From Corollary~\ref{subcubic} we get $s+2d\le n$ and $4s+2d \le m$. Therefore, we get two upper bounds for $|H|$.

$$\setlength\arraycolsep{0.2em}
 \begin{array}{rclclcl}
|H| & = & s + d & \le & s + \frac{n-s}{2}  & = &  \frac{n}{2} + \frac{s}{2}\\
|H| & = & s + d & \le & s + \frac{m-4s}{2} & = & \frac{m}{2} - s
 \end{array}
$$


\noindent This implies $3|H|\le n+\frac{m}{2}$.
\end{proof}

\begin{corollary}
 Let $G$ be a graph on $n$ vertices and $m$ edges. Then $$\psi_3(G)\le \frac{ 2n+m}{6}.$$
\end{corollary}

\begin{theorem}
 Let $a,b$ are integers such that $b\le a \le 2b$. There 
is a graph on $n$ vertices and $m$ edges such that $\frac{m}{n}=\frac{a}{b}$ and $$\psi_3(G)\ge \frac{2n+m}{6}.$$
\end{theorem}

\begin{proof}
Let $x=3(2b-a)$ and $y=2(a-b)$.
We construct the graph $G$ having two types of components:
\begin{itemize}
 \item $x$ components $C_4$ (a cycle on $4$ vertices) with 4 edges, $\psi_3(C_4)=2$.
 \item $y$ components $H_6$ ($K_6$ with a perfect matching removed) with 12 edges, $\psi_3(H_6)=4$.
\end{itemize} 
It is easy to see that $n=4x+6y$ and $m=4x+12y$. Then 
$$\frac{2n+m}{6}=\frac{12x+24y}{6}=2x+4y=\psi_3(G).$$
Moreover, 
$$\frac{m}{n}=\frac{2x+6y}{2x+3y}=\frac{6(2b-a)+12(a-b)}{6(2b-a)+6(a-b)}=\frac{6a}{6b}=\frac{a}{b}.$$
\end{proof}


\end{document}